\documentclass[12pt]{article}
\usepackage{latexsym,amsfonts,amsmath,theorem,amssymb,bbm}
\usepackage{graphicx}
\usepackage{verbatim}
\usepackage{tikz,tikz-cd}
\usepackage{enumitem}
\scrollmode

\newtheorem{theorem}{Theorem}
\newtheorem{lemma}[theorem]{Lemma}
\newtheorem{corollary}[theorem]{Corollary}

\newtheorem{remark}[theorem]{Remark}
\newenvironment{proof}{\begin{trivlist}
    \item[\hskip\labelsep{\bf Proof.}]}{$\hfill\Box$\end{trivlist}}

\newcommand{\N}{\mathbb{N}}

\newcommand{\mask}[1]{}

\DeclareMathOperator{\Prob}{{\bf P}}
\DeclareMathOperator{\Expec}{{\bf E}}

\DeclareMathOperator{\bj}{{\bf j}}
\DeclareMathOperator{\bi}{{\bf i}}
\DeclareMathOperator{\bk}{{\bf k}}
\DeclareMathOperator{\bx}{{\bf x}}
\DeclareMathOperator{\bl}{{ \bf l}}

\DeclareMathOperator{\kl}{K_{low}}

\DeclareMathOperator{\ku}{K_{up}}

\DeclareMathOperator{\supp}{{supp}}

\title{Randomized sparse grid algorithms for multivariate integration on {H}aar-Wavelet spaces}
\author{M. Wnuk, M. Gnewuch}
\date{\today}

\begin{document}
\maketitle

\begin{abstract}
The \emph{deterministic} sparse grid method, also known as Smolyak's method, is a well-established and widely used tool to tackle multivariate approximation problems, and there is a vast literature on it. Much less is known about \emph{randomized} versions of the sparse grid method.
In this paper we analyze randomized sparse grid algorithms, namely randomized sparse grid quadratures  for multivariate integration on the $D$-dimensional unit cube $[0,1)^D$. Let $d,s \in \N$ be such that $D=d\cdot s$. The $s$-dimensional building blocks of the sparse grid quadratures are based on stratified sampling for $s=1$ and on scrambled $(0,m,s)$-nets for $s\ge 2$. The spaces of integrands and the error criterion we consider are Haar wavelet spaces with parameter $\alpha$ and the randomized error (i.e., the worst case root mean square error), respectively. 
We prove sharp (i.e., matching) upper and lower bounds for the convergence rates of the $N$-th mininimal errors for all possible combinations of the parameters $d$ and $s$. 


 Our upper error bounds still hold if we consider as spaces of integrands Sobolev spaces of mixed 
dominated smoothness with smoothness parameters $1/2< \alpha < 1$ instead of Haar wavelet spaces.
\end{abstract}

\section{Introduction}\label{Smolyak2SECT1}

The Smolyak method, also known as sparse grid method, was introduced by Smolyak in \cite{Smo63} and is a widely used method to 
approximate the solution of tensor product problems. Its wide range of applicability is described e.g. in the survey articles \cite{BG04, Gri06},  book chapters \cite[Chapter~15]{NW10}, \cite[Chapter~4]{DTU18}, and the book \cite{DS89}.

In particular, Smolyak algorithms are helpful when it comes to integration problems in moderate or high dimension; references include, e.g., \cite{Gen74, Del90, Tem92, BD93, FH96, NR96, GG98, GG03, Pet03, GLSS07, HHPS18,PW04}.

All references mentioned above study  \emph{deterministic} Smolyak algorithms.
Randomization of algorithms allows for a statistical error analysis and may help to achieve better convergence rates.
Unfortunately, there is much less known about \emph{randomized} Smolyak algorithms. 
The only references we are aware of are \cite{DLP07, HM11} and our previous paper
\cite{GW19}. 

In \cite{DLP07} Dick et al. investigate a specific instance of the randomized Smolyak method
for the integration problem in certain Sobolev spaces over $[0,1)^D$, $D=d\cdot s$.
Firstly,  they consider higher-order digital sequences in $[0,1)^s$ with a fixed digital shift and use them as building block algorithms for a deterministic $D$-dimensional Smolyak construction. They prove that the resulting integration  algorithms achieve almost optimal order of convergence (up to logarithmic factors) of the worst-case error.
Secondly, they randomize their construction by chosing the digital shift randomly and achieve for the resulting randomized Smolyak construction a corresponding  convergence order  of the mean square worst-case error (which is obtained by averaging the square of the worst-case error over all digital shifts and then taking the square root).
In \cite{HM11} Heinrich and Milla employ the randomized Smolyak method as a building block of an algorithm computing antiderivatives of functions from $L^p([0,1]^d)$ allowing for fast computation of antiderivative values for any point in $[0,1]^d.$ Note that in both cases the randomized Smolyak method is applied in specific cases and none of the papers gives a systematic treatment of its properties.

In \cite{GW19} we started a systematic analysis of the randomized Smolyak method. In that paper we studied general linear approximation problems on tensor products of Hilbert spaces.
We considered two error criteria, namely the root mean square worst-case error (as has been done in \cite{DLP07}) and the worst-case root mean square error (or, shortly, \emph{randomized error}),
and provided for both criteria upper and lower error bounds for randomized Smolyak algorithms.
(Note that the latter error criterion is more commonly used to assess the quality of randomized algorithms than the former one.) In our upper error bounds the dependence of the constants on the underlying number of variables as well as on the number of sample points (or, more general, information evaluations) used is made explicit.

In this follow-up article we consider the integration problem over the domain $[0,1)^D$, where $D=d\cdot s$ for some $d,s\in \N$. The spaces of integrands are certain Haar-wavelet spaces, consisting of square integrable functions whose Haar-wavelet coefficients decay fast enough. 
We analyze the performance of randomized Smolyak algorithms that use as building block algorithms equal weight quadratures based on scrambled $(0,m,s)$-nets in some base $b$. The considered error criterion is the randomized error. We pursue mainly two aims: 

Firstly, we want to show that under more specific assumptions on the approximation problem we may improve on the general upper error bounds for randomized Smolyak algorithms provided in \cite{GW19}. (This phenomenon is also known for deterministic Smolyak algorithms, see, e.g., \cite{WW95, GLSS07, NW10}.) ``More specific'' means here that we consider an integration problem (instead of a general linear approximation problem) and a specific reproducing kernel Hilbert space defined via the decay of Haar-wavelet coefficients (instead of a general tensor product Hilbert space). 
We are able to show via a lower error bound (see Theorem~\ref{Thm:LowerBound}) 
that our upper error bound provided in Theorem~\ref{UpperBoundTheorem} is sharp, i.e., that our Smolyak algorithms based on scrambled nets cannot achieve a higher rate of convergence. This means we have determined the exact asymptotic convergence rate of the randomized error of our algorithms.

Secondly, we want to make a precise comparison between randomized quasi-Monte Carlo (RQMC) algorithms based on scrambled $(0,m,s)$-nets (that can be seen as specialized algorithms, tuned for the specific integration problem) on the one hand and Smolyak algorithms based on one-dimensional nets (which can be seen as special instances of a universal method) on the other hand.

This comparison is motivated by the comparison done in \cite{DLP07}. There the authors study Smolyak algorithms for all possible combinations of $d$ and $s$ with $d\cdot s = D$, starting with $s=1, d=D$, i.e., Smolyak algorithms based on one-dimensional nets, up to $s=D, d=1$, i.e., pure higher-order nets in dimension $D$. Their error bounds get larger for larger $d$ by essentially a factor of $(\log N)^{d\cdot  \alpha}$,
where $N$ is the number of sample points employed by the Smolyak method and $\alpha$ is the degree of smoothness of the considered Sobolev space, see \cite[Corollary~4.2]{DLP07}.
Since Dick et al. did not provide matching lower bounds (and we actually believe that their upper bounds are not optimal with respect to the logarithmic order), the question remains whether this really gives a faithful picture of the changes in the convergence order of the error. Moreover, the error criterion used in \cite{DLP07}, the root mean square worst-case error, is a rather untypical error criterion for randomized algorithms. 

That is why we perform a comparison of Smolyak algorithms based on scrambled $(0,m,s)$-nets for all possible combinations of $d$ and $s$ with $d\cdot s = D$ in the Haar-wavelet space setting with respect to the commonly used randomized error criterion. Due to Heinrich et al. \cite{HHY03} the exact convergence order of QMC quadratures based on scrambled $(0,m,s)$-nets is known, and our results determine the exact convergence order of Smolyak algorithms based on scrambled nets. 
We can deduce that the convergence order of the randomized error decreases for increasing $d$ exactly by a factor $(\log N)^{(d-1)(1+\alpha)}$, where again 
$N$ is the number of sample points employed by the algorithm at hand and $\alpha > \frac{1}{2}$ is the ``degree of smoothness'' of the considered Haar-wavelet space.
This factor describes the loss in the convergence order we have to accept in this setting if we use for integration a universal tool as Smolyak's method 
instead of the more specialized, optimal tool, namely scrambled $(0,m,D)$-nets.

The paper is organized as follows: At the beginning of Section \ref{Smolyak2SECT2} we give a general overview of the Smolyak method discussed in this article. Then in Subsection \ref{Nets} we recall the properties of $(0,m,s)$-nets and explain what kind of randomization we are considering. Subsection \ref{FunctionSpaces} gives a short account of Haar-wavelets spaces, which constitute the input space of the algorithms considered. Subsection \ref{UB} contains the main result of this article, an upper bound on the integration error of the Smolyak algorithm on Haar wavelet spaces, see Theorem \ref{UpperBoundTheorem}. As explained in Corollary \ref{Cor:Error_Sobolev}, for $1/2 < \alpha < 1$ the same upper bound holds also for the integration error of the algorithm on Sobolev spaces of dominated mixed smoothness $\alpha$. Finally, in Subsection \ref{LB} we show that the previously established upper bound is optimal by providing a matching lower bound (Theorem \ref{Thm:LowerBound}). Theorem \ref{UpperBoundTheorem} and Theorem \ref{Thm:LowerBound} allow now for an easy comparison between the different variants of the randomized Smolyak algorithm that correspond to different values of $s$ and $d.$

We finish the introduction with a few remarks on notation. If not explicitly stated otherwise all additions and comparisons involving vectors and real numbers are performed coordinatewise, so e.g. for $\bk = (k_1    ,\ldots, k_d) \in \mathbb{N}^d , \bj = (j_1    ,\ldots, j_d) \in \mathbb{N}^d$ the formula $\bk = \bj -1$ is to be read as $k_n = j_n - 1, n = 1,\ldots, d.$

In the sequel $b \in \mathbb{N}_{\geq 2}$ stands for the base of the considered $(0,m,s)$-nets and is fixed. For $j \in \mathbb{N}_0$ put $\theta_j = \{0\}$ if $j = 0$ and $\theta_j = \{0,\ldots,b-1\}$ else. For $\bj \in \mathbb{N}^t_0$ set
$$\theta_{\bj} = \theta_{j_1} \times \cdots \times \theta_{j_t}.$$
Similarly put $\vartheta_j = \{0\}$ for $j \in \{-1,0\}$ and $\vartheta_j = \{0,\ldots, b^j - 1\}$ for $j \in \mathbb{N}.$ For $\bj \in \mathbb{Z}_{\geq -1}^{t}$ set
$$\vartheta_{\bj} = \vartheta_{j_1} \times \cdots \times \vartheta_{j_t}.$$

Finally, for a vector ${\bf{v}} = (v_1,\ldots, v_t)$ we write $|{\bf{v}}| := \sum_{j = 1}^t | v_j|.$

\section{Multivariate Integration on Haar-wavelet spaces: Smolyak vs. $(0,m,s)$-Nets}\label{Smolyak2SECT2}
We are considering the problem of integrating $D-$variate functions belonging to Haar-wavelet spaces with smoothness parameter $\alpha > \frac{1}{2}.$  The solution operator is the integration operator $I_{D},$ which is given by
$$I_{D}: \mathcal{H}^{D}_{ \alpha} \rightarrow \mathbb{R}, \quad I_{D}f = \int_{[0,1)^{D}}f(\bx) \, d\bx.$$
The $D-$variate Haar-wavelet space $\mathcal{H}^{D}_{\alpha}$ with smoothness parameter $\alpha$ consists of all square-integrable functions whose wavelet coefficients converge to $0$ quickly enough and  will be defined later in Section \ref{FunctionSpaces}.
We would like to use the randomized Smolyak method with building blocks being scrambled $(0,m,s)-$nets, where $D = ds$ for some $d \in \mathbb{N}.$
The building blocks of our Smolyak method $((U_l^{(n)})_{l \in \mathbb{N}})_{n \in [d]}$ are randomized QMC quadratures such that for every $n,$ $U^{(n)}_l$ is based on a scrambled $(0,l-1,s)$- net in base $b$, i.e. for some scrambled $(0,l-1,s)-$net $(Y_i)_{i = 1}^{b^{l-1}}$ and any function $f$  we have
$$U^{(n)}_{l} f = b^{-l+1} \sum_{i = 1}^{b^{l-1}} f(Y_i).$$
For the definition of $(0,m,s)-$nets in base $b$ see Section \ref{Nets}. 
Moreover, set $U^{(n)}_0, n = 1,\ldots, d,$ to be the algorithm returning a.s. $0$ for every input.
We make the standing assumption that all the building blocks $U^{(n)}_l, n = 1,\ldots, d, l \in \mathbb{N},$ are randomized independently.
Let $L,d \in \mathbb{N}$ with $L \geq d.$ The randomized $d-$variate Smolyak method of level $L$ is constructed as in \cite{GW19},
i.e. denoting
\begin{equation}
  Q(L,d) := \left\{\bl \in \mathbb{N}^d \, | \, | \bl | \leq L\right\},
\end{equation}
and
\begin{equation}\label{DeltaDef}
   \Delta_l^{(n)} := U_l^{(n)} - U_{l-1}^{(n)}, \quad  l \in \mathbb{N}, n = 1,\ldots, d,
\end{equation}
we put
\begin{equation} \label{SmolyakAlg}
A(L,d) = \sum_{\bl \in Q(L,d)} \bigotimes_{n = 1}^d \Delta_{l_n}^{(n)}.
\end{equation}
The randomized Somolyak method lives on the probability space $(\Omega, \Sigma, \Prob),$ where $\Omega = \Omega^{(1)} \times \ldots \Omega^{(d)},$ $\Sigma = \bigotimes_{n = 1}^d \Sigma^{(n)}$ and $\Prob = \bigotimes_{n = 1}^d \Prob^{(n)},$ and $(\Omega^{(n)}, \Sigma^{(n)}, \Prob^{(n)})$ is the probability space on which the sequence $(U^{(n)}_l)_l$ lives. The product structure reflects the fact that the quadratures from the different coordinates of the Smolyak algorithm are randomized independently, cf. \cite{GW19}. 
It can be verified  that the following representation
\begin{equation}\label{alg_bb_rep}
A(L,d) = \sum_{ L-d+1 \leq |\bl| \leq L} (-1)^{L-|\bl|} \binom{d-1}{L-|\bl|} \bigotimes_{n = 1}^d U^{(n)}_{l_n},
\end{equation}
is valid,
cf. \cite[Lemma 1]{WW95}.
Our focus lies on the \emph{randomized (root mean square) error} $e^{{\rm{r}}}(I_D,A(L,d))$ of  the described Smolyak method, which is given by
$$e^{{\rm{r}}}(I_D,A(L,d)) = \sup_{\lVert f \rVert_{\mathcal{H}^{D}_{\alpha}} \leq 1} \left(   \Expec \left[ |A(L,d)f - I_Df |^2  \right] \right)^{\frac{1}{2}}.$$
We denote by $N$ the number of function evaluations performed by the Smolyak method. Our aim is to show that there exist constants $c,C > 0$ independent of $N$ such that
$$  c \frac{(\log(N))^{(d-1)(1+\alpha)}}{N^{\alpha + \frac{1}{2}}} \leq  e^{{\rm{r}}}(I_D,A(L,d)) \leq C \frac{(\log(N))^{(d-1)(1+\alpha)}}{N^{\alpha + \frac{1}{2}}}. $$
 Both bounds, the upper and the lower one, are obviously sharp and improve on the bounds which we could obtain by direct application of 
 \cite[Theorem 3]{GW19} by logarithmic factors (note however, that the bounds from \cite{GW19} are much more general since they can be applied to any linear approximation problem in a Hilbert space setting). Changing the parameter $d$ from $1$ to $D$ through the divisors of $D$ yields a whole family of algorithms, ranging from quadratures based purely on scrambled $(0,m,D)$-nets to a randomized Smolyak method with building blocks stemming from one-dimensional nets (i.e. one-dimensional stratified sampling).

The following lemma holds actually for general randomized Smolyak quadratures, and the first statement is valid also for the deterministic ones.

\begin{lemma}\label{WeightsSum}
Let $d, L \in \mathbb{N}, L \geq d,$ and $(U^{(n)}_l)_{l \in \mathbb{N}}, n = 1,\ldots, d,$ be sequences of quadratures on $[0,1)^s$ such that weights used by every $U^{(n)}_l$ sum up to $1.$
Then
\begin{enumerate}
\item The weights used by the Smolyak method $A(L,d)$ based on the building blocks $(U_l^{(n)})_l, n = 1,\ldots, d,$ also sum up to $1,$ i.e. for any $f :[0,1)^{ds} \rightarrow \mathbb{R}$ we have
$$A(L,d)f = \sum_{\nu = 1}^N w_{\nu} f(x_{\nu}), $$
for some sampling points $x_{\nu} \in [0,1)^{ds},$ and weights $w_{\nu} \in \mathbb{R}$ with $\sum_{\nu = 1}^N w_{\nu} = 1,$
\item If additionally $f \in L^{1}([0,1)^{ds}) $ and all the sampling points employed by $(U_l^{(n)})_{l,n}$ are uniformly distributed in $[0,1)^s$  then $A(L,d)f$ is an unbiased estimator of the integral of $f.$
\end{enumerate}  
\end{lemma}
\begin{proof}
The second statement follows immediately from the first one. We prove the first statement.
 For arbitrary quadratures $Q^{(1)}, \ldots, Q^{(d)},$ such that
$$Q^{(n)}f = \sum_{j = 1}^{N_n} w^{(n)}_j f(x_j^{(n)}), \hspace{2ex}, n = 1, \ldots, d, $$
denoting $w^{(n)} := \sum_{j = 1}^{N_n} w_j^{(n)},$ we have that the sum of the weights used by the quadrature $Q := \bigotimes_{n = 1}^d Q^{(n)}$ is equal to $\prod_{n = 1}^d w^{(n)}.$ Now, the sum of the weights used by $\Delta^{(n)}_{l_n}$ (see (\ref{DeltaDef})) is equal to $0$ if $l_n > 1$ and is equal to $1$ if $l_n = 1.$ It follows from (\ref{SmolyakAlg}) that the sum of weights used by $A(L,d)$ is equal to the sum of weights used by $\bigotimes_{n = 1}^d \Delta^{(n)}_1,$ i.e. to $1.$
\end{proof}

\begin{remark}
\begin{enumerate}
\item One may easily find examples where weights applied by the Smolyak method are negative, so in particular $A(L,d)$ is not a (randomized) QMC quadrature.
\item Note that utilizing representation (\ref{alg_bb_rep}) one may also find an alternative formula for the sum of weights of our Smolyak algorithm, namely
$$\sum_{ \nu = 0}^{\min\{d-1, L-d \}} (-1)^{\nu} \binom{d-1}{\nu} \binom{L-\nu - 1}{d-1}.$$
It follows that the above sum is equal to $1.$
\end{enumerate}
\end{remark}

\begin{figure}
    \centering
    \begin{minipage}{0.45\textwidth}
        \centering
        \includegraphics[width=0.9\textwidth]{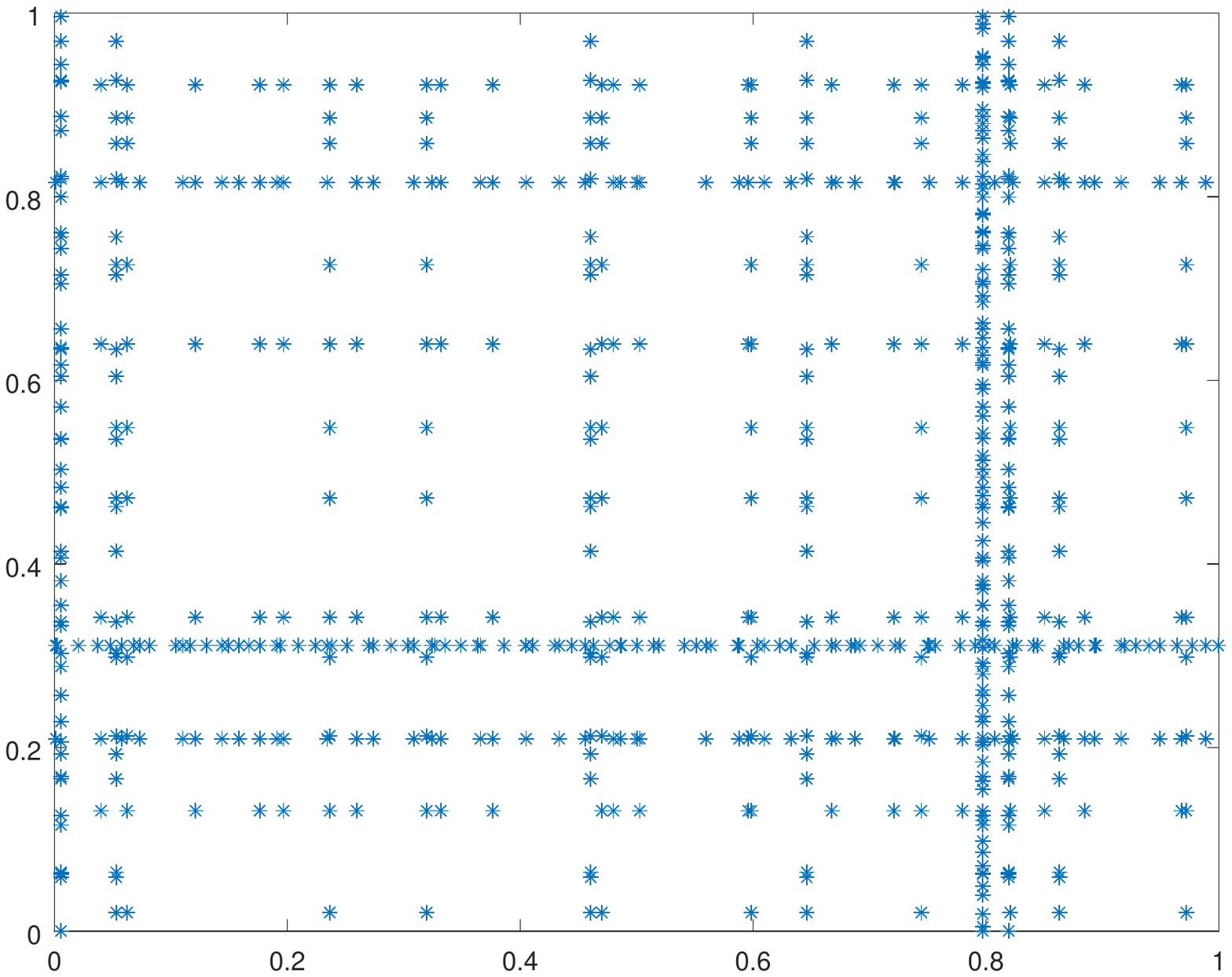} 
        \caption{Smolyak integration nodes in case $s = 1, d = 2.$ $N = 640$. Note that the weights used are not equal.}
    \end{minipage}\hfill
    \begin{minipage}{0.45\textwidth}
        \centering
        \includegraphics[width=0.9\textwidth]{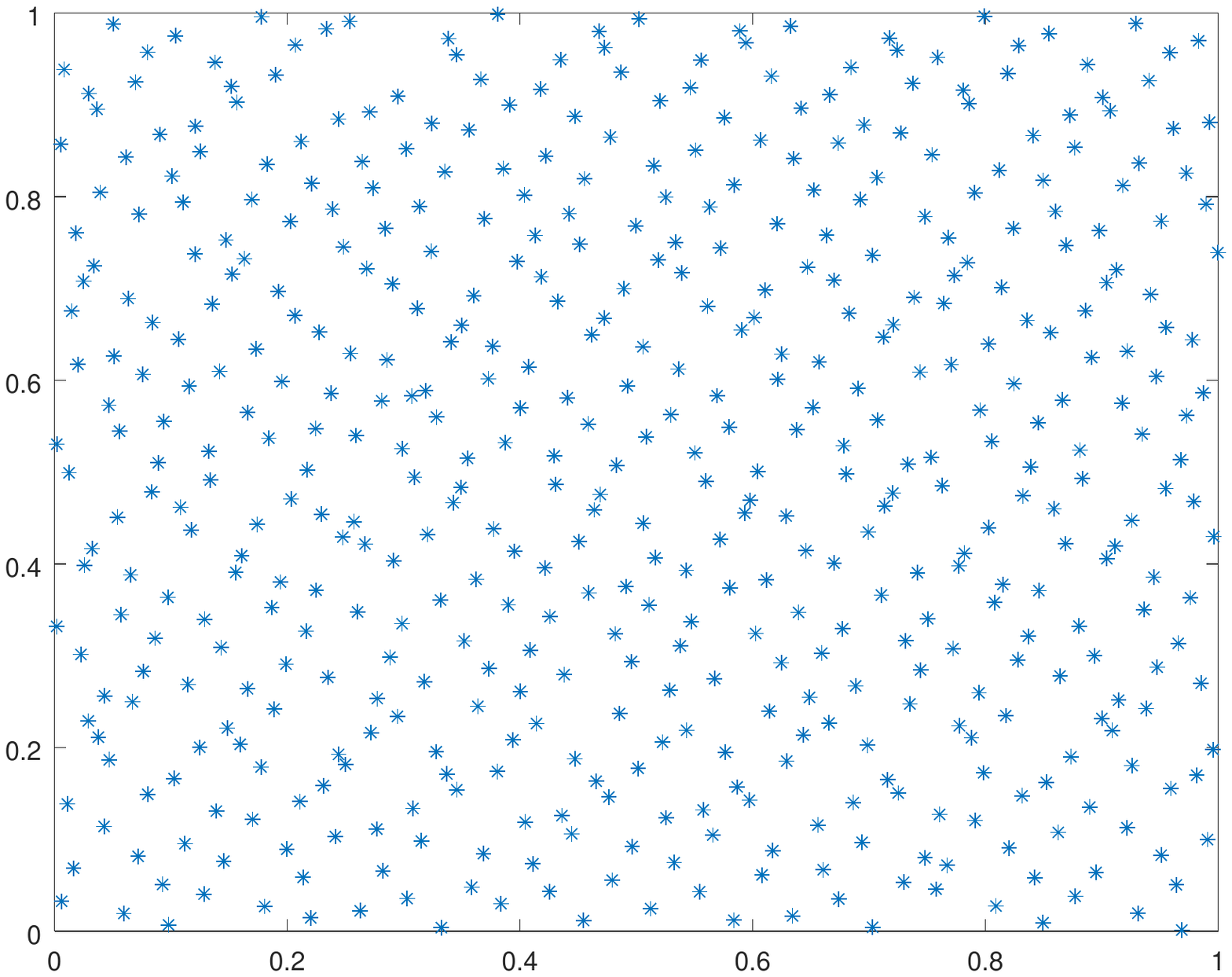} 
        \caption{Smolyak integration nodes in case $s=2, d=1,$ scrambled Sobol points. $N = 512$.}
    \end{minipage}
\end{figure}

\subsection{Scrambled ($0,m,s)-$Nets} \label{Nets}
Fix $b \in \mathbb{N}_{\geq 2}.$ For $j \in \mathbb{N}_0$ and $k \in \vartheta_j$ we define the one-dimensional \emph{elementary interval (in base $b$)} $E^j_k$ to be
$$E^j_k := [kb^{-j},(k+1)b^{-j}).$$
Moreover, we put
$$E^{-1}_k  := [0,1).$$
Let $s \in \mathbb{N}.$   An $s-$dimensional interval is called \emph{elementary} if for some $\bj \in \mathbb{N}^s_0$ and $\bk \in \vartheta_{\bj}$ it is of the form
$$E^{\bj}_{\bk} := E^{j_1}_{k_1} \times \cdots \times E^{j_s}_{k_s}.$$

Let $m \in \mathbb{N}.$ A finite set $\mathcal{P} \subset [0,1)^s$ is called a $(0,m,s)-$\emph{net} in base $b$ if each elementary $s-$dimensional interval of volume $b^{-m}$ contains exactly one point of $\mathcal{P}.$ The notion  of $(0,m,s)-$nets was introduced by Niederreiter in \cite{Nie87}. Nets may be build using e.g. the Faure construction \cite{Fau81, Fau82}. 

The most common way of obtaining randomized nets is via scramblings, which was introduced by Owen, see \cite{Owe95, Owe97}. Here we present a slight modification of the original idea, see \cite[Section 2.4.]{Mat10}.

A mapping $\sigma: [0,1) \rightarrow [0,1)$ is called a \emph{$b-$ary \emph{scrambling} (of depth $l$)} if it is constructed in the following way. Let $x =_{b} 0.x_1x_2\ldots x_m x_{m+1}\ldots$ be a $b-$ary representation of $x$ (whenever possible we choose the representation having finitely many digits different from $0$). To determine $\sigma(x)$ we first fix some permutation $\pi$ of $\{0,1,\ldots, b-1\}$ and let the first $b-$ary digit of $\sigma(x)$ be $y_1 = \pi(x_1).$ Next for every possible value of  $x_1 $ we fix a permutation $\pi_{x_1}$ of $\{0,1,\ldots, b-1\}$ and define the second $b-$ary digit of $\sigma(x)$ to be $y_2 = \pi_{x_1}(x_2).$ Continuing analogously, for every $x_1, x_2$ we choose a permutation $\pi_{x_1, x_2}$ and let the third digit of $\sigma(x)$ be $y_3 = \pi_{x_1, x_2}(x_3).$ In this way we get the first $l$ digits of $\sigma(x).$ There are many approaches concerning how to choose the reminder term. We fix all the other digits of $\sigma(x)$ to be equal $0.$ We speak of a random scrambling when all the permutations are chosen independently from the uniform distribution on permutations of $\{0,1,\ldots, b-1\}.$
Let now $\mathcal{P} \subset [0,1)^s$ be a $(0,m,s)-$ net in base $b$ and let $\boldsymbol{\sigma} :=(\sigma_1,\ldots, \sigma_s)$ be a vector of independent random scramblings of depth $l \geq m.$ We define the set
$$\boldsymbol{\sigma}(\mathcal{P}) := \{ (\sigma_1(p_{i,1}) + b^{-l}\xi_{i}^{1}, \ldots , \sigma_s(p_{i,s}) + b^{-l}\xi_{i}^{s}) \, | \, (p_{i,1},\ldots, p_{i,s}) =: p_i \in \mathcal{P}  \}, $$
where the family of reminder terms $(\xi_{i}^{t}), i \in \mathbb{N}, t \in [s],$ consists of iid random variables distributed according to Unif([0,1)) and independent of all the permutations.
 If $\mathcal{P}$ is a $(0,m,s)-$net the same holds true for $\boldsymbol{\sigma}(\mathcal{P}).$ Moreover, every single $p \in \boldsymbol{\sigma}(\mathcal{P})$ is distributed according to Unif$([0,1)^s)$. We call the set $\boldsymbol{\sigma}(\mathcal{P})$ a \emph{scrambled net}.



For every $l \in \mathbb{N}$ let  $\boldsymbol{\sigma}^{(n)}_{l} := (\sigma^{(n)}_{l,t})_{t \in [s]}$ be a vector of independent scramblings of depth $l-1.$ We require all the permutations used by $((\sigma^{(n)}_{l,t})_{t \in [s]})_{n \in [d], l \in \mathbb{N}}$ and all the reminder terms $((\xi^{(n),t}_{l,i})_{t \in [s], i \in [b^{j-1}]})_{n \in [d], l \in \mathbb{N}}$ to be independent.
We let $U_l^{(n)}$ be a QMC quadrature based on $\boldsymbol{\sigma}^{(n)}_{l}(\mathcal{P}^{(n)}_l)$ for some $(0,l-1,s)$-net $\mathcal{P}^{(n)}_l.$ We say in this case that all the building blocks $U^{(n)}_l, n = 1,\ldots,d, l \in \mathbb{N},$ are scrambled independently.  

\subsection{Function Spaces}\label{FunctionSpaces}

Now we define Haar-wavelet spaces $\mathcal{H}_{\alpha}^{D}$. We start with univariate wavelets. Let $b \in \mathbb{N}_{ \geq 2}$ (it will later coincide with the base of a $(0,m,s)-$net) and define
$$\psi_{i}(x) = b^{\frac{1}{2}} \mathbbm{1}_{\lfloor bx \rfloor = i } - b^{-\frac{1}{2}}\mathbbm{1}_{\lfloor x \rfloor = 0 }, \quad i = 0,1,\ldots, b-1,$$
here $\lfloor x \rfloor$ denotes the largest integer smaller or equal then $x$ and for a statement $S$ the characteristic function $\mathbbm{1}_S$ takes on value $1$ if $S$ is true and $0$ otherwise.
For integers $j > 0$ and $k \in \vartheta_{j-1}$ put
$$\psi^j_{i,k}(x) := b^{\frac{j-1}{2}} \psi_i(b^{j-1}x - k) = b^{\frac{j-2}{2}}[b \mathbbm{1}_{\lfloor b^{j}x \rfloor = bk + i } - \mathbbm{1}_{\lfloor b^{j-1} x \rfloor = k } ].$$
Note that $\psi_i = \psi^{1}_{i,0}.$ Moreover, define
$$\psi^{0}_{0,0} := \mathbbm{1}_{[0,1)}. $$
We shall refer to those functions as \emph{Haar wavelets}.
The parameters have the following interpretation: $j$ describes the resolution of the wavelet, which corresponds to the size of its support, $i$ describes the shape and $k$ is the shift parameter. We have $\supp(\psi^{j}_{i,k}) = E^{j-1}_{k}.$
The wavelets are normalized with respect to the $L^2$-norm, i.e.
$$\Vert \psi^{j}_{i,k} \rVert_{L^2} = 1.$$
For an integer $j \geq 0$ let us put
$$V_j := \{f \in L^2([0,1)) \, | \, f \text{ is constant on intervals } E^j_k, k \in \vartheta_{j}\}. $$
Moreover, we define $W_0 = V_0$ and $W_{j+1}$ to be the orthogonal complement of $V_j$ in $V_{j+1}$, i.e. we have
$$V_{j+1} = V_j \oplus W_{j+1} = \bigoplus_{\nu = 0}^{j+1} W_{\nu} .$$
One may easily see that
$$W_j = {\rm{span}}\{\psi^j_{i,k} \, | \, i \in \theta_j, k \in \vartheta_{j-1}\} $$
and $\bigoplus_{j = 0}^{\infty} W_j$ is dense in $L^2([0,1))$.

Now we proceed to multivariate wavelets which are nothing more then tensor products of univariate wavelets.
Given vectors $\bj \in \mathbb{N}^{D}_0, \bi \in \theta_{\bj}, \bk \in \vartheta_{\bj - 1}$ we write
$$ \Psi^{D, \bj}_{\bi, \bk}(\bx) := \prod_{t=1}^{D} \psi^{j_t}_{i_t, k_t}(x_t), \quad \bx = (x_1,\ldots, x_D) \in [0,1)^{D}.$$ 
For $r \in \mathbb{N}, L \in \mathbb{N}_0$ let
$$J(r, L) := \left\{ \bj \in \mathbb{N}^{r}_0 \, | \, |\bj| = L \right\}.$$
We define the $s$-variate approximation space of level $L$ to be
$$V^{s, L} := \sum_{\bj \in J(s, L)} \bigotimes_{t = 1}^{s} V_{j_t}.$$
Moreover, let
$$W^{s, L} := \bigoplus_{\bj \in J(s,L)} \bigotimes_{t = 1}^{s} W_{j_t}. $$
Note that the space $W^{s, L}$ is represented as an orthogonal sum whereas $V^{s, L}$ is not.
It is easily verified that 
$$W^{s,L} = {\rm{span}}\{\Psi^{s, \bj}_{\bi, \bk} \, | \, \bj \ \in J(s,L), \bi \in \theta_{\bj}, \bk \in \vartheta_{\bj - 1} \},$$
$V^{s,L}$ consists of linear combinations of functions constant on all the $s$-dimensional elementary intervals in base $b$ of volume $b^{-L},$
and 
$$V^{s,L+1} = V^{s, L} \oplus W^{s, L+1}.$$
Finally, we put
$$V^{D,L} := \sum_{\bj \in J(d,L)} \bigotimes_{n = 1}^d V^{s, j_n},$$
and 
$$W^{D,L} := \bigoplus_{\bj \in J(d,L)} \bigotimes_{n = 1}^{d} W^{s, j_n}.$$


Note that the wavelets $(\Psi^{D, \bj}_{\bi, \bk})_{\bj, \bi, \bk}$  are not orthogonal but they do integrate to $0$ as long as at least one coordinate of $\bj$ is positive.
We may represent every function $f \in L^2([0,1)^D)$ as

\begin{equation}\label{WavRep}
f  = \sum\limits_{\bj, \bk, \bi} \hat{f}^{\bj}_{\bi, \bk}\Psi^{D, \bj}_{\bi, \bk} :=  \sum\limits_{\bj}\sum\limits_{\bk}\sum\limits_{\bi} \hat{f}^{\bj}_{\bi, \bk}\Psi^{D, \bj}_{\bi, \bk}
\end{equation}
with Haar coefficients
$$\hat{f}^{\bj}_{\bi, \bk}  = \int_{[0,1)^{D}} f(\bx) \Psi^{D, \bj}_{\bi, \bk}(\bx) \, d\bx$$
where the convergence holds in $L^2$ sense (for details, see \cite{Owe97}). Functions $(\Psi^{D, \bj}_{\bi, \bk})_{\bj, \bk,\bi}$  form a tight frame in $L^2([0,1)^{D}),$ meaning that there exists a constant $c$ such that for every $f \in L^2([0,1)^{D})$ we have $\sum\limits_{\bj, \bi, \bk} |\langle f, \Psi^{D, \bj}_{\bi, \bk}\rangle_{L^2}|^2 = c \lVert f \rVert^2_{L^2}$ (in this special case we even have $c=1$).

For $\alpha > \frac{1}{2}$ we consider the following \emph{(D-variate) Haar-wavelet space (with smoothness $\alpha$)}
$$\mathcal{H}_{\alpha}^{D} := \left\{f \in L^2([0,1)^D) \, \bigg| \, \sum\limits_{\bj, \bk, \bi} b^{2\alpha|\bj|} (\hat{f}^{\bj}_{\bi, \bk})^2 < \infty \right\}.$$
We equip $\mathcal{H}^{D}_{\alpha}$ with the inner product
$$\langle f,g  \rangle_{\mathcal{H}_{\alpha}^{D}}  :=  \sum\limits_{\bj, \bk, \bi} b^{2\alpha|\bj|} \hat{f}^{\bj}_{\bi, \bk} \hat{g}^{\bj}_{\bi, \bk}.$$

\begin{remark}\label{Rem:Sobolev}
Let $1/2 < \alpha < 1$. It is known that the Sobolev space $H^{\alpha} [0,1] \subseteq L^2([0,1))$ is continuously embedded into $\mathcal{H}_{\alpha}^1$, see, e.g., \cite{GLSS07} for the case where $b=2$ or \cite{DGMS} for a more general statement. Since the Sobolev space of dominated mixed smoothness $H^{\alpha}_{\rm mix} ( [0,1]^D)$ and  the Haar wavelet space $\mathcal{H}_{\alpha}^D $ are the $D$-fold tensor product Hilbert spaces of $H^{\alpha} [0,1]$ and $\mathcal{H}_{\alpha}^1$, respectively, also $H^{\alpha}_{\rm mix} ( [0,1]^D)$ is continuously embedded in $\mathcal{H}_{\alpha}^D$.
\end{remark}

Scaled wavelets and scaled coefficients are denoted by $$\Psi^{D, \bj, \alpha}_{\bi, \bk} := b^{-\alpha|\bj| } \Psi^{D, \bj}_{\bi, \bk}$$ and $$\hat{f}^{\bj, \alpha}_{\bi, \bk} := b^{ \alpha |\bj|} \hat{f}^{\bj}_{\bi, \bk}.$$
Note that the scaled wavelets $\Psi^{D, \bj, \alpha}_{\bi, \bk}$ are normalized with respect to the $\mathcal{H}^{D}_{\alpha}$ norm.
We want to represent a function $f \in \mathcal{H}_{\alpha}^{D}$ as a vector of wavelet coefficients.
For $n \in \mathbb{N}_0$ put
$$J(n) := \{(\bj, \bk, \bi) \in \mathbb{N}_{0}^{D} \times \mathbb{N}_{0}^{D} \times \mathbb{N}_{0}^{D} \, | \, |\bj| = n, \bk \in \vartheta_{\bj - 1}, \bi \in \theta_{\bj}\}.$$
Let $\hat{f}^{n}$ be a vector consisting of the wavelet coefficients of $f$ pertaining to indices from $J(n)$ ordered in some fixed way. Put
$$\hat{f} :=  (\hat{f}^n)_{n \in \mathbb{N}_0}. $$
In a similar fashion we define $\hat{f^{\alpha}}.$ Ordering the functions  $\Psi^{D,\bj}_{\bi,\bk}$ in exactly the same way we obtain vectors of functions $\Psi$ and $\Psi^{\alpha}$. Using the above notation we may write identity (\ref{WavRep}) as $$f = (\hat{f^{\alpha}})^{T} \Psi^{\alpha}.$$ Here ${\bf{v}}^{T}$ is just the transposed vector $\bf{v}.$

\subsection{Upper Bounds on the Integration Error }\label{UB}
We want to investigate the problem of integrating $D = sd$-variate functions from the Haar-wavelet space $\mathcal{H}^{D}_{\alpha}$. To this end we put 
$$I_s: \mathcal{H}_{\alpha}^{s} \rightarrow \mathbb{R}, \quad I_sf = \int_{[0,1)^s}f(\bx) \, d\bx,$$
and assume for every  $n=1,\ldots,d, l \in \mathbb{N},$ and almost every $\omega \in \Omega^{(n)}$ that $U^{(n)}_l(\omega):\mathcal{H}_{\alpha}^{s} \rightarrow \mathbb{R}$, is a quadrature based on a scrambled $(0,l-1,s)$-net as described in Section \ref{Nets}. Furthermore, $A(L,d)$ is the Smolyak method constructed as in (\ref{SmolyakAlg}).

To make notation easier to follow for a vector $\bf{v}$ we introduce $$u({\bf{v}}) := \{ r \, | \,  v_r \neq 0\}.$$

We call the sets $$B_a = \{as+1, as+2,\ldots (a+1)s\}, \hspace{3ex} a=0,1\ldots, d-1, $$ \emph{blocks}. Given $\bj \in \mathbb{N}_0^{D}$ we say that the block $B_a$ is \emph{active (with respect to $\bj$)} if $B_a \cap u(\bj) \neq \emptyset.$ With the obvious notation we have
$$ \Psi^{D, \bj}_{\bi, \bk} := \bigotimes_{n = 1}^{d} \Psi^{s, \bj_n}_{\bi_n, \bk_n}.$$
Note that $\supp(\Psi^{s,\bj_n}_{\bi_n, \bk_n}) = E^{\bj_n - 1}_{\bk_n} \subseteq [0,1)^s.$

\begin{lemma}\label{RedDim}
Consider $\bj, \bi, \bk \in \mathbb{N}^{D}_0$ such that $\Psi^{D, \bj}_{\bi, \bk} = \Psi^{D-s, \bj'}_{\bi', \bk'} \otimes \Psi^{s,0}_{0,0}$  for some $\bj',\bi',\bk' \in \mathbb{N}_0^{D - s}$.      Then
$$A(L,d)  \Psi^{D, \bj}_{\bi, \bk} = A(L-1,d-1)  \Psi^{(d-1)s,\bj'}_{\bi', \bk'}.  $$
\end{lemma}
A similar lemma was proved for deterministic Smolyak algorithms on multiwavelet spaces, see \cite[Lemma 4.3.]{GLSS07}
\begin{proof}
Straightforward calculations yield
\begin{align*}
A(L,d) \Psi^{D, \bj}_{\bi, \bk}
&=\sum_{\bl \in Q(L,d)}\bigotimes_{n=1}^{d-1} \Delta^{(n)}_{l_n} \Psi^{s,\bj_n}_{\bi_n,\bk_n} \otimes \Delta^{(d)}_{l_d} \mathbbm{1}_{[0,1)^s} \\
&=\sum_{\bl \in Q(L-1,d-1)}\bigotimes_{n=1}^{d-1} \Delta^{(n)}_{l_n} \Psi^{s,\bj_n}_{\bi_n,\bk_n} \otimes \Delta^{(d)}_{1} \mathbbm{1}_{[0,1)^s} \\
&= \sum_{\bl \in Q(L-1,d-1)}\bigotimes_{n=1}^{d-1} \Delta^{(n)}_{l_n} \Psi^{s,\bj_n}_{\bi_n,\bk_n} \\
&= A(L-1,d-1) \Psi^{(d-1)s,\bj'}_{\bi', \bk'} .
\end{align*}
\end{proof}

By simple induction we get the following corollary.

\begin{corollary} \label{DimRed}
Consider $\bj, \bi, \bk \in \mathbb{N}^{D}_0$ such that the function $\Psi^{D,\bj}_{\bi, \bk}$ has $d-t > 0$ active blocks. Let  $\bj', \bi', \bk' \in \mathbb{N}^{(d-t)s}$ be the vectors $\bj, \bi, \bk$ where the coordinates from the inactive $t$ blocks are removed. Then
$$A(L,d)  \Psi^{D,\bj}_{\bi, \bk} = A(L-t,d-t)  \Psi^{(d-t)s,\bj'}_{\bi', \bk'} $$
\end{corollary}

Recall that a randomized quadrature $Q$ is \emph{exact} if almost surely $Qf$ is equal to the integral of $f$ for all functions $f$ from the input space. Note that exact randomized quadratures are automatically unbiased.  

\begin{lemma} \label{Exact2}
 Let $L,d,s,l \in \mathbb{N}, L \geq d$ and $D = ds.$ For $n = 1,\ldots, d,$ we have
 \begin{enumerate}
 \item $U^{(n)}_l$ is exact on $V^{s,l-1}$
 \item $A(L,d)$ is exact on $V^{D, L-d}.$
 \end{enumerate}
 \end{lemma}
 Similar statements as Lemma \ref{Exact2} (2) have been provided for deterministic Smolyak algorithms over different function spaces e.g. in \cite[Theorem 2]{NR96} and \cite[Theorem 4.1.]{GLSS07}.
\begin{proof}[Proof of Lemma \ref{Exact2}]
The first statement follows immediately by the definitions of $(0,l-1,s)$-nets and $V^{s,l-1}.$ We prove the second statement.
Since the algorithm is linear for almost every $\omega \in \Omega$ it suffices to prove the exactness of $A(L,d)$ on all the wavelets $\Psi^{D, \bj}_{\bi, \bk}, |\bj| \leq L-d$. Fix one such function. The case $\Psi^{D, \bj}_{\bi, \bk} \equiv 1$ follows by Lemma \ref{WeightsSum}.
 Suppose now that all the blocks of $\Psi^{D, \bj}_{\bi, \bk}$ are active.  We are using the representation (\ref{alg_bb_rep}) 
\begin{equation}
A(L,d) = \sum_{ L-d+1 \leq |\bl| \leq L} (-1)^{L-|\bl|} \binom{d-1}{L-|\bl|} \bigotimes_{n = 1}^d U^{(n)}_{l_n}.
\end{equation}
By the first statement, for $l'>l$ the algorithm $U_{l'}^{(k)}$ is exact on $V^{s,l}.$
Take any ${\bf l}$ over which we are summing up in the above formula. Since $\sum_{n = 1}^d l_n \geq L-d+1$ and $\sum_{n = 1}^d |\bj_n| \leq L-d$ there exists at least one index $\mu$ for which $l_{\mu} > |\bj_{\mu}|$ and so $U^{(\mu)}_{l_{\mu}} \Psi^{s, \bj_{\mu}}_{\bi_{\mu}, \bk_{\mu}} = 0.$ Hence $A(L,d)\Psi^{D, \bj}_{\bi, \bk} = 0,$ which is the exact value of the integral of $\Psi^{D, \bj}_{\bi, \bk}.$

If $\bj$ admits $t < d$ inactive blocks we may use Corollary \ref{DimRed} and the above argument applied to $A(L-t,d-t).$
\end{proof}



\begin{lemma}\label{ZeroOutDiag}
Let  $(U^{(n)}_l)_{l \in \mathbb{N}}$ for every $n =1,\ldots, d,$ be a sequence of independent quadratures based on scrambled nets in base $b.$ 
If $\bj \neq \bj'$ or if $\bk \neq \bk',$ then for every $l,l' \in \mathbb{N}$

\begin{equation} \label{CrossTerms1}
\Expec\left[U^{(n)}_l \Psi^{s,\bj}_{\bi, \bk} U^{(n)}_{l'} \Psi^{s,\bj'}_{\bi', \bk'} \right] = 0.
\end{equation}

As a result, if either $\bj \neq \bj'$ or $\bk \neq \bk',$  for any $L \in \mathbb{N}$ the Smolyak method $A(L,d)$ based on building blocks $(U^{(n)}_l)_{l \in \mathbb{N}, n \in [d]}$  satisfies

\begin{equation} \label{CrossTerms2}
\Expec \bigg[A(L,d) \Psi^{D, \bj}_{\bi, \bk} A(L,d) \Psi^{D, \bj'}_{ \bi', \bk'}\bigg]  = 0.
\end{equation}

\end{lemma}
\begin{proof}

If $l \neq l'$, identity \eqref{CrossTerms1} follows from the independence of $U^{(n)}_l$ and $U^{(n)}_{l'}.$ If $l = l'$ then \eqref{CrossTerms1} follows from \cite[Lemma 4]{Owe97}. Now \eqref{CrossTerms2} may be easily deduced using representation \eqref{alg_bb_rep}. 

\end{proof}

\begin{lemma}\label{ExactBoundBB}
There exist positive constants $c_{b,s}, C_{b,s} $ depending only on $b,s,$ such that for every $l \in \mathbb{N}$ and $\bj \in \mathbb{N}^s_0$ with $|\bj| \geq l + s - 1$ and every admissible $\bi, \bk$ we have 
$$c_{b,s} b^{-l} \leq \Expec\left[ (U_l \Psi^{s, \bj}_{\bi, \bk})^2 \right] \leq C_{b,s} b^{-l},$$
where $U_l$ denotes an RQMC quadrature based on a scrambled $(0,l-1,s)$-net  in base $b$.
\end{lemma}
\begin{proof}

For $\Psi^{s, \bj}_{\bi, \bk}$ as above and $x \in \supp(\Psi^{s, \bj}_{\bi, \bk})$ we  have
$$\Psi^{s, \bj}_{\bi, \bk}(x) = \prod_{t = 1}^{s} b^{\frac{j_t - 2}{2}} \gamma_t(x),$$
where $\gamma_t(x) \in \{ b-1, -1, b \}$. It follows that  $|\Psi^{s, \bj}_{\bi, \bk}(x)| =  b^{\frac{|\bj|}{2}  - s}\gamma(x),$ with $\gamma(x) = \prod_{t = 1}^s |\gamma_t(x)|.$ Put $\overline{\gamma} = \lambda^s(\supp(\Psi^{s, \bj}_{\bi, \bk}))^{-1} \int_{\supp(\Psi^{s, \bj}_{\bi, \bk})} \gamma(x)^2 \, dx$. Note that $1\leq \overline{\gamma} \leq b^{2s}.$ Now, since $|\bj| \geq l+s-1$, at most one point of the net used by $U_l$ falls into $\supp(\Psi^{s, \bj}_{\bi, \bk})$ and this happens with probability $b^{l - |\bj|} \xi(\bj),$ where $1 \leq \xi(\bj) \leq b^{s-1}.$ Denoting by $(p_t)_{t = 1}^{b^{l-1}}$ the scrambled net used by $U_l$ we get
\begin{align*}
&\Expec \left[  (U_l\Psi^{s, \bj}_{\bi, \bk})^2   \right]  = b^{-2l + 2} \Expec \left[  \left( \sum_{t = 1}^{b^{l-1}}  \Psi^{s, \bj}_{\bi, \bk}(p_t)  \right)^2    \right]
\\
& = b^{-2l+2}  b^{l - |\bj|} \xi(\bj) b^{|\bj| - 2s} \overline{\gamma}.
\end{align*}
It follows
$$b^{2-2s} b^{-l} \leq  \Expec\left[ (U_l \Psi^{s, \bj}_{\bi, \bk})^2 \right]   \leq b^{1+s}b^{-l} $$
\end{proof}

\begin{lemma} \label{MatEntBound}
Let $\Psi^{D, \bj}_{\bi, \bk}$, $\Psi^{D, \bj}_{\bi', \bk}$ be two (not necessarily different)  wavelets and for $L \geq d$ let $A(L,d)$ be the Smolyak method as in (\ref{SmolyakAlg}), where all the building blocks $U^{(n)}_l, n = 1,\ldots, d, l \in \mathbb{N},$ are randomized independently. Then there exists a constant $C = C(b,d,s)$ which does not depend on $L, \bj$ nor $\bk$, such that
$$\left|\Expec\bigg(  \big[   (I_D - A(L,d))   \Psi^{D, \bj}_{\bi, \bk}      \big] \big[ (I_D - A(L,d))\Psi^{D, \bj}_{\bi', \bk}     \big]                  \bigg) \right| \leq C|\bj|^{d-1}b^{-L}. $$
\end{lemma}
\begin{proof}
The case $u(\bj) = \emptyset$ is trivial so without loss of generality we may assume $u(\bj) \neq \emptyset.$ Since $u(\bj) \neq \emptyset$  we have $I_D\Psi^{D, \bj}_{\bi, \bk} = 0$ and so  $$\Expec\left(  \left[   (I_D - A(L,d))   \Psi^{D, \bj}_{\bi, \bk}      \right] \left[ (I_D - A(L,d))\Psi^{D, \bj}_{\bi', \bk}     \right]                  \right) = \Expec \left[ A(L,d) \Psi^{D, \bj}_{\bi, \bk} A(L,d)   \Psi^{D, \bj}_{\bi', \bk}\right].$$
 By the Cauchy-Schwarz inequality
$$\left|\Expec \bigg[ A(L,d) \Psi^{D, \bj}_{\bi, \bk} A(L,d)   \Psi^{D, \bj}_{\bi', \bk}\bigg] \right| \leq \Expec\bigg[(A(L,d)  \Psi^{D, \bj}_{ \bi, \bk})^2  \bigg]^{\frac{1}{2}}\Expec \bigg[(A(L,d)  \Psi^{D, \bj}_{\bi', \bk})^2  \bigg]^{\frac{1}{2}},$$
so it suffices to show $\Expec\bigg[(A(L,d)  \Psi^{D, \bj}_{\bi, \bk})^2\bigg]  \leq C|\bj|^{d-1}b^{-L}. $

We shall now proceed by induction on $(L,d)$ starting with $(L,1)$ for some $L \in \mathbb{N},$ and then moving from $(L-1,d-1)$ to $(L,d).$
For the induction start recall that $A(L,1)$ is just a scrambled $(0,L-1,s)$-net quadrature. In the case $L=1$ we have
$$\Expec \left[ (A(1,1) \Psi^{D, \bj}_{\bi, \bk})^2  \right] = \lVert \Psi^{D, \bj}_{\bi, \bk} \rVert_{L^2}^2 = 1 .$$
Let now $L > 1.$ Due to Lemma \ref{Exact2} the algorithm  $A(L,1)$ is exact on $ \Psi^{s, \bj}_{\bi, \bk}$ for $|\bj| \leq L-1$ . Let now $|\bj| > L-1.$ By Lemma 3 from \cite{HHY03}
$$\Expec\bigg[(A(L,1) \Psi^{s, \bj}_{\bi, \bk})^2\bigg] = N(L,1)^{-1} \Gamma_{ \bj} \leq b^{1-L}, $$
where $N(L,1) = b^{L-1}$ is the number of evaluation points used and $\Gamma_{ \bj}$  is an appropriate gain coefficient which by Lemma $6$ from \cite{HHY03}  is bounded from above by $1.$ This proves the claim for $d=1.$

Let us note that in the induction step we may confine ourselves to wavelets with all the blocks active. Indeed, let $t \in \{1,2,\ldots, d-1\}$ and $\bj$ be such that $\Psi^{D, \bj}_{\bi, \bk}$ admits $d-t$ active blocks. Corollary \ref{DimRed} and our induction hypothesis yield
\begin{align*}
\begin{split}
 &\Expec\bigg[(A(L,d) \Psi^{D, \bj}_{\bi, \bk})^2\bigg] = \Expec\bigg[(A(L-t,d-t)\Psi^{(d-t)s,\bj'}_{\bi', \bk'} )^2\bigg] \leq \\
& C(b,d-t,s)|\bj'|^{d-t-1}b^{-L+t} \leq C(b,d,s)|\bj|^{d-1} b^{-L},
\end{split}
\end{align*}
where $\bj',\bi',\bk'$ have the same meaning as in Corollary \ref{DimRed} and $C(b,d,s) = b^tC(b,d-t,s).$
Now we may go over to the proof of the induction step assuming that we always deal only with functions with all the blocks active.

Note that for many ${\bf l} \in \mathbb{N}_{0}^d$ we have a.s. that $\left(\bigotimes_{n = 1}^d U^{(n)}_{l_n}\right) \Psi^{D, \bj}_{\bi, \bk} = 0,$ and all the vectors ${\bf l}$ for which this is not the case are contained in the set $S_{\bj} = \bigcup_{\mu \geq d}S_{\bj, \mu},$ where $S_{\bj, \mu}$ is defined as
$$S_{\bj, \mu} := \{{\bf l}\in \mathbb{N}^d \, | \, |{\bf l}| = \mu, \forall_{n = 1,\ldots, d} \, l_n \, \leq |\bj_n| \}.$$
A justification of the inclusion is the following: Suppose that for some $t$ we have $l_t > |\bj_t|.$ Integrating every (active) block yields $0,$ thus exactness of $U^{(t)}_{l_t}$ on $\Psi^{s,\bj_t}_{\bi_t, \bk_t}$ gives $U^{(t)}_{l_t}\Psi^{s,\bj_t}_{\bi_t, \bk_t} = 0$.

 Note that
$$|S_{\bj, \mu}| \leq \binom{|{\bf j}| - \mu + d - 1}{d-1};$$
this is because, basically, one asks, in how many ways one can distribute the 'overshoot' of $|{\bf j}| - \mu $ elements between $d$ coordinates.


By Lemma \ref{ExactBoundBB} we have for a constant $C_{b,s}$ not depending on $l$
$$\Expec[(U^{(n)}_l \Psi^{s,\bj}_{\bi,\bk})^2] \leq C_{b,s} b^{-l}.$$
Calling upon the representation (\ref{alg_bb_rep}) and using the independence of the building blocks we finally obtain with a generic constant $C$:
\begin{align*}
\begin{split}
\Expec\bigg[(A(L,d) \Psi^{D,\bj}_{\bi,\bk})^2\bigg]& = \Expec\left[\left( \sum_{L-d+1 \leq|{\bf l}| \leq L} (-1)^{L - |{\bf l}|} \binom{d-1}{L - |{\bf l}|}  \bigotimes_{n = 1}^d U^{(n)}_{l_n} \Psi^{s,\bj_n}_{\bi_n,\bk_n}  \right)^2\right]
\\
&\leq C \sum_{\mu = L-d+1}^L  \sum_{{\bf l} \in S_{\bj, \mu}}\prod_{n = 1}^{d} \Expec \left[\left(U^{(n)}_{l_n} \Psi^{s,\bj_n}_{\bi_n,\bk_n} \right)^2\right]
\\
&\leq C \sum_{{\bf l} \in S_{\bj, L-d+1}} b^{-L} 
\\
&\leq C\binom{|\bj| + 2d - 2 - L}{d-1}b^{-L} \leq C |\bj|^{d-1} b^{-L}
\end{split}
\end{align*}
\end{proof}

Recall the definition of the vector $\Psi^{\alpha}$ from Section \ref{FunctionSpaces}.
For a randomized quadrature $A$ we put
\begin{equation}\label{Lambda}
\Lambda_{A} := \Expec\bigg[((I_D - A)\Psi^{\alpha})((I_D - A)\Psi^{\alpha})^T\bigg],
\end{equation}
where the operator $(I_D - A(\omega))$ is applied to the vector $\Psi^{\alpha}$ componentwise.
It has been shown in \cite[Theorem 1]{HHY03} that if $A$ is a randomized quadrature applied to integrands from $\mathcal{H}^{D}_{\alpha}$, one has
$$e^{\rm{r}}(I_D,A) = \sqrt{\rho(\Lambda_{A})},$$
where $\rho(\Lambda_{A})$ denotes the spectral radius of $\Lambda_{A}$, i.e. the largest absolute value of eigenvalues of $\Lambda_A.$ 
Put
\begin{equation}\label{Lambda2}
 \Lambda := \Lambda_{A(L,d)}.
\end{equation}

\begin{remark}\label{LambdaStructure}
By Lemma \ref{ZeroOutDiag} $\Lambda$ is a block diagonal matrix consisting of $b^{|u(\bj)|} \times b^{|u(\bj)|}$ blocks $\Lambda(\bj, \bk)$ given by
$$\Lambda(\bj, \bk) = \Expec[((I_D - A(L,d)) \Psi^{D,\bj, \alpha}_{\bi, \bk})_{\bi \in \theta_{\bj}} ((I_D - A(L,d)) \Psi^{D,\bj, \alpha}_{\bi, \bk})^T_{\bi \in \theta_{\bj}} ]. $$
For $u(\bj) = \emptyset$ we have $\Lambda(\bj,\bk) = 0 \in \mathbb{R}^{1 \times 1}.$ For $u(\bj)  \neq \emptyset$ we obtain $I_D\Psi^{D,\bj, \alpha}_{\bi, \bk} = 0 $ and consequently
$$\Lambda(\bj, \bk) = \Expec[(A(L,d) \Psi^{D,\bj, \alpha}_{\bi, \bk})_{\bi \in \theta_{\bj}} (A(L,d) \Psi^{D,\bj, \alpha}_{\bi, \bk})^T_{\bi \in \theta_{\bj}} ]. $$
 Due to the block diagonal structure the eigenvalues of $\Lambda$ are exactly the eigenvalues of all the blocks, i.e.
 \begin{align*}
  e^{\rm{r}}(I_D,A(L,d))^2 & =\sup_{\bj, \bk} \rho(\Lambda(\bj, \bk)).
\end{align*}
\end{remark}

\begin{remark}\label{SmolyakCardinality}
Consider $(\widetilde{U}^{(n)}_l), n = 1,\ldots, d,$ and $l \in \mathbb{N},$ where $\widetilde{U}^{(n)}_l$ is a (randomized) QMC quadrature based on a (scrambled) $(0,l,s)$-net. Let $\widetilde{A}(L,d)$ be the corresponding Smolyak method. The building blocks $\widetilde{U}^{(n)}_l$ satisfy the assumption (20) from \cite{GW19} with $\kl = 1, \ku = 2$ and $K = b.$ It follows by \cite[Lemma 8]{GW19} that the cardinality of information used by $\widetilde{A}(L,d)$ is of the order $\Theta(b^LL^{d-1}).$ Now, changing the level of the nets used by the building blocks by $1$ ($U^{(n)}_l$ is based on a scrambled $(0,l-1,s)$-net instead of a $(0,l,s)$-net) influences only the implicit constant, i.e. also the cardinality of information used by $A(L,d)$ is of the order $\Theta(b^LL^{d-1}).$ 
\end{remark}

\begin{theorem}\label{UpperBoundTheorem}
Let $\alpha > \frac{1}{2}, s,d \in \mathbb{N}, D = sd$. For $n = 1,\ldots, d,$ and $l \in \mathbb{N}$ let $U^{(n)}_l$ be a randomized QMC quadrature based on a scrambled $(0,l-1,s)$-net in base $b$. Let $A(L,d)$ denote the corresponding Smolyak method approximating the integral $I_D: \mathcal{H}^{D}_{\alpha} \rightarrow \mathbb{R},$  where the building blocks $(U_l^{(n)})_{l,n}$ are all randomized independently, and denote by $N$  the number of function evaluations performed by $A(L,d)$. Then there exists a constant $C$ independent of $N$ such that the randomized error satisfies

\begin{equation}\label{uppererrorbound}
e^{\rm{r}}(I_D, A(L,d)) \leq C \frac{\log(N)^{(d-1)(1 + \alpha)}}{N^{\alpha + \frac{1}{2}}}.
\end{equation}
\end{theorem}

\begin{remark}
Note that using the optimal bounds for errors of quadratures based on scrambled $(0,m,s)-$nets from \cite{HHY03} and then applying the general Theorem 9 from \cite{GW19} yields a weaker asymptotic bound than \eqref{uppererrorbound} with exponent $(d-1)(\frac{3}{2} + \alpha)$ in the logarithmic term.
\end{remark}

\begin{corollary}\label{Cor:Error_Sobolev}
Let $1/2 < \alpha < 1$. Since the Sobolev space of dominated mixed smoothness 
$H^{\alpha}_{\rm mix} ( [0,1]^D)$ is continuously embedded in $\mathcal{H}_{\alpha}^D $ (cf. Remark~\ref{Rem:Sobolev}),
Theorem \ref{UpperBoundTheorem} still holds if we replace the space of integrands $\mathcal{H}_{\alpha}^D $ by $H^{\alpha}_{\rm mix} ( [0,1]^D)$ and measure the randomized error in the Sobolev norm instead of the Haar wavelet norm.

That is, the upper error bound \eqref{uppererrorbound} for the randomized Smolyak method holds also for multivariate integration over the Sobolev space of mixed dominated smoothness $H^{\alpha}_{\rm mix} ( [0,1]^D)$.
\end{corollary}


\begin{proof}[Proof of Theorem \ref{UpperBoundTheorem}]
By Lemma \ref{MatEntBound} the entries of $\Lambda(\bj, \bk)$ are of the order $b^{-2\alpha |\bj|} |\bj|^{d-1} b^{-L}$. Due to the exactness result from  Lemma \ref{Exact2}  we only have to consider blocks for which $|\bj| > L - d \approx L$, since for other $\bj$ the blocks $\Lambda(\bj, \bk)$ consist just of zeros.
  For any square matrix $D = (d_{i,j})_{i,j = 1}^n$ all the eigenvalues  lie inside $\bigcup_{i = 1}^{n}K_i$, where $$K_i = \{z \in \mathbb{C} \, | \, |z - d_{i,i}| \leq \sum_{j \neq i} |d_{i,j}|\}$$
  are Gershgorin circles.
Due to Remark \ref{LambdaStructure} we have, for some suitable constants $C$ and $\tilde{C},$
\begin{align*}
e^{\rm{r}}(I_D,A(L,d))^2 & =\sup_{\bj, \bk} \rho(\Lambda(\bj, \bk))
 \\
  &\leq \sup_{\bj : |\bj| > L-d} C b^{|u(\bj)|} b^{- 2 \alpha |\bj|}|\bj|^{d-1}b^{-L}
  \\
    &\leq \sup_{\bj: |\bj| > L-d} C b^{D} b^{- 2 \alpha |\bj|}|\bj|^{d-1}b^{-L}
  \\
   &\leq \tilde{C} b^{-L(1+2\alpha)} L^{d-1}.
\end{align*}
To justify the last inequality it is enough to show that for some $\tilde{C}$ not depending on $L$ we have
\begin{equation}\label{boundJust}
\sup_{\bj : |\bj| > L-d} b^{2\alpha (L-|\bj|)} \left(\tfrac{|\bj|}{L}\right)^{d-1} \leq \tilde{C}.
\end{equation}
To this end consider $f: \mathbb{R}_{>0} \rightarrow \mathbb{R}, x \mapsto b^{2 \alpha(L - x)} \left(\tfrac{x}{L}\right)^{d-1}.$
It holds
$$f'(x) = b^{2 \alpha(L - x)} \left( \tfrac{x}{L} \right)^{d-2} \left[ -2\alpha  \log(b)\tfrac{x}{L} + \tfrac{d-1}{L}\right],$$
and so $f$ attains its (sole) local maximum at $x_0 = \tfrac{d-1}{2 \alpha \log(b)}.$ 
Therefore
\begin{align*}
& \sup_{\bj : |\bj| > L-d} b^{2\alpha (L-|\bj|)} \left(\tfrac{|\bj|}{L}\right)^{d-1} \leq b^{2 \alpha d} \left( \max \left\{\tfrac{d-1}{2 \alpha \log(b) L}, \frac{L-d}{L} \right\}   \right)^{d-1}
\\
& \leq b^{2 \alpha d} \left( \tfrac{1}{2\alpha \log(b)}\right)^{d-1} \leq (2b)^{2 \alpha d},
\end{align*}
which proves the desired inequality. 

Remark \ref{SmolyakCardinality} yields
$$N = \Theta(b^L L^{d-1}).$$
It follows
$$b^{-L(1 + 2\alpha)}L^{d-1} = \frac{L^{(d-1)(2 + 2\alpha)}}{b^{L(1 + 2\alpha)} L^{(d-1)(1+2\alpha)}} = \mathcal{O} \bigg( \bigg( \frac{(\log(N))^{(d-1)(1 + \alpha)}}{N^{\alpha + \frac{1}{2}}}  \bigg)^2 \bigg).$$
Hence we have
$$e^{\rm{r}}(I_D,A(L,d)) \leq C \frac{(\log(N))^{(d-1)(1 + \alpha)}}{N^{\alpha + \frac{1}{2}}}.$$
\end{proof}

\subsection{Lower Bounds on the Integration Error}\label{LB}

\begin{lemma}\label{ArrangementsLowerBound}
 Let $d, L \in \mathbb{N}, \bj = \left( \lfloor \frac{2L}{d} \rfloor, \lfloor \frac{2L}{d} \rfloor, \ldots, \lfloor \frac{2L}{d} \rfloor, 2L - (d-1)\lfloor \frac{2L}{d} \rfloor \right) \in \mathbb{N}^d$ and $$S_{\bj, L} =  \{{\bf l}\in \mathbb{N}^d \, | \, |{\bf l}| = L, \forall_{n = 1,\ldots, d} \, l_n \, \leq |\bj_n| \}.$$
 One has
 $$|S_{\bj, L}| = \Omega(L^{d-1}),$$
 where the implicit constant may depend on $d.$
\end{lemma}
\begin{proof}
The case $d = 1$ is trivial, so let us assume $d \geq 2.$
 For the ease of presentation we shall prove only the case when $d$ divides $2L.$ After obvious changes the proof goes through in the general case. So let $\bj = \left( \frac{2L}{d}, \ldots, \frac{2L}{d}\right).$ Moreover, since we are interested in asymptotic statement for $L,$ we may without loss of generality assume that $L \geq 2d(d-1)$. One may describe the situation in the following way: we have $d$ bins numbered $1,2,\ldots, d$, each of them with capacity $ \frac{2L}{d}.$ How many ways are there to arrange $L$  indistinguishable balls in those bins? Let us focus on the bins numbered $1, \ldots,d-1$. Any of them may have any number of balls between $\frac{(d-2)L}{d(d-1)} + 1$ and $\frac{dL}{d(d-1)} - 1$ independently of all the other bins  numbered $1,\ldots, d-1.$ Indeed, if every of those bins has $\frac{(d-2)L}{d(d-1)} + 1$ balls then putting in the $d$-th bin $\frac{2L}{d} - (d-1)$ balls we end up with $L$ balls altogether. On the other hand, if every of the first $d-1$ bins has $\frac{dL}{d(d-1)}-1$ balls then putting $d-1$ balls in the $d$-th bin we also end up with $L$ balls altogether. That is,
 $$|S_{\bj, L}| \geq \left( \frac{dL - (d-2)L}{d(d-1)} - 1 \right)^{d-1} = \left( \frac{2L}{d(d-1)} - 1    \right)^{d-1} = \Omega(L^{d-1}). $$
\end{proof}

\begin{theorem}\label{Thm:LowerBound}
Let $\alpha > \frac{1}{2}, d,s \in \mathbb{N}, D = sd$. For $n = 1,\ldots d,$ and $l \in \mathbb{N},$ let $U^{(n)}_l$ be independently randomized  RQMC quadratures based on scrambled $(0,l-1,s)$-nets. Let $A(L,d)$ denote the corresponding Smolyak method approximating the integral $I_D: \mathcal{H}^{D}_{\alpha} \rightarrow \mathbb{R}.$ Then there exists a constant $c_{b,s}$ independent of $N$ such that the randomized error satisfies
$$e^{\rm{r}}(I_D,A(L,d)) \geq c_{s,b} \frac{(\log(N))^{(d-1)(1 + \alpha)}}{N^{\alpha + \frac{1}{2}}}.$$
\end{theorem}

\begin{proof}
Recall the definition of $\Lambda$ from (\ref{Lambda2}). Remark \ref{SmolyakCardinality} gives $N = \Omega(b^{L}L^{d-1}).$ Due to Remark \ref{LambdaStructure} it suffices to show that
\begin{equation}\label{LBcondition}
\sup_{\bj, \bk} \rho(\Lambda(\bj,\bk)) \geq c b^{-L(1+2 \alpha)}L^{d-1},
\end{equation}
 with a constant $c > 0$ not depending on $L$, where $\Lambda(\bj, \bk)$ is the block of $\Lambda$ corresponding to $\Psi^{D, \bj}_{\bi, \bk}, \bi \in \theta_{\bj}.$

To ascertain that $\sup_{\bj, \bk}\rho(\Lambda(\bj, \bk)) \geq c b^{-L(1+ 2\alpha)} L^{d-1},$ it is enough to show that for some $\tilde{c} > 0$ independent of $L$ and some $\bj, \bk,$ some diagonal entry of $\Lambda(\bj, \bk)$ is bounded from below by $\tilde{c} b^{-L(1+ 2\alpha)} L^{d-1}.$ Indeed, if this holds then the same bound holds also for the trace of $\Lambda(\bj, \bk)$ and for $\rho(\Lambda(\bj, \bk)).$ That is, we only need to show that for an appropriate $\bj$ with $|\bj|$ proportional to $L$ and some admissible $\bi, \bk$ we have
$$\Expec\left[ \left(  A(L,d) \Psi^{D, \bj}_{\bi, \bk}  \right)^2   \right] \geq \tilde{c} b^{-L}L^{d-1}.$$ 

 We take $ \bj = \left( \lfloor \frac{2L}{d} \rfloor, \lfloor \frac{2L}{d} \rfloor, \ldots, \lfloor \frac{2L}{d} \rfloor, 2L - (d-1)\lfloor \frac{2L}{d} \rfloor \right).$ Recall from Lemma \ref{ArrangementsLowerBound} the definition of $S_{\bj, L}$
and that $|S_{\bj, L}| = \Omega(L^{d-1}).$
Hence, we obtain for $L$ large enough
\begin{align*}
& \Expec\left[ \left( A(L,d) \Psi^{D, \bj}_{\bi, \bk}  \right)^2   \right] =  \Expec\left[\left( \sum_{L-d+1 \leq|{\bf l}| \leq L} (-1)^{L - |{\bf l}|} \binom{d-1}{L - |{\bf l}|}  \bigotimes_{n = 1}^d U^{(n)}_{l_n} \Psi^{s,\bj_n}_{\bi_n,\bk_n}  \right)^2\right]
\\
& = \Expec \left[ \sum_{\mu = L - d +1}^{L}  \sum_{{\bf l} \in S_{\bj, \mu}} \binom{d-1}{L- \mu}^2 \left(\bigotimes_{n = 1}^d U^{(n)}_{l_n} \Psi^{s, \bj_n}_{\bi_n, \bk_n}    \right)^2 \right]
\\
& \geq C \sum_{\mu = L - d +1}^{L} |S_{\bj, \mu}|b^{-\mu} \geq C|S_{\bj, L}| b^{-L} \geq C L^{d-1}b^{-L},
\end{align*}
where the first inequality follows by Lemma \ref{ExactBoundBB} and the second inequality by Lemma \ref{ArrangementsLowerBound}.
\end{proof}

\bibliographystyle{siam}
\bibliography{References}

\begin{thebibliography}{10}

\bibitem{BD93}
{\sc G.~Baszenski and F.~J. Delvos}, {\em Multivariate {B}oolean midpoint
  rules}, in Numerical Integration IV, H.~Brass and H.~H{\"a}mmerlin, eds.,
  Basel, 1993, Birkh{\"a}user, pp.~1--11.

\bibitem{BG04}
{\sc H.~J. Bungartz and M.~Griebel}, {\em Sparse grids}, Acta Numerica, 13
  (2004), pp.~147--269.

\bibitem{Del90}
{\sc F.-J. Delvos}, {\em {B}oolean methods for double integration}, Math.
  Comp., 55 (1990), pp.~683--692.

\bibitem{DS89}
{\sc F.-J. Delvos and W.~Schempp}, {\em Boolean Methods in Interpolation and
  Approximation}, vol.~230 of Pitman Research Notes in Mathematics, Longman,
  Essex, 1989.

\bibitem{DGMS}
{\sc J.~Dick, M.~Gnewuch, L.~Markhasin, and W.~Sickel}, {\em Optimal cubature
  rules on {H}aar wavelet spaces and fractional discrepancy}.
\newblock Preprint, 2020.

\bibitem{DLP07}
{\sc J.~Dick, G.~Leobacher, and F.~Pillichshammer}, {\em Randomized {S}molyak
  algorithms based on digital sequences for multivariate integration}, IMA J.
  Numer. Analysis, 27 (2007), pp.~655--674.

\bibitem{DTU18}
{\sc D.~D{\~u}ng, V.~Temlyakov, and T.~Ullrich}, {\em Hyperbolic Cross
  Approximation}, Birkh{\"a}user, Basel, 2018.

\bibitem{Fau81}
{\sc H.~Faure}, {\em Discr\'epance de suites associ\'ees \`a un syst\'eme de
  num\'eration (en dimension un)}, Bull. Soc. Math. France, 109 (1981),
  pp.~143--182.

\bibitem{Fau82}
\leavevmode\vrule height 2pt depth -1.6pt width 23pt, {\em Discr\'epance de
  suites associ\'ees \`a un syst\'eme de num\'eration (en dimension $s$)}, Acta
  Arith., 41 (1982), pp.~338--351.

\bibitem{FH96}
{\sc K.~Frank and S.~Heinrich}, {\em Computing discrepancies of {S}molyak
  quadrature rules}, J. Complexity, 12 (1996), pp.~287--314.

\bibitem{Gen74}
{\sc A.~C. Genz}, {\em Some extrapolation methods for the numerical calculation
  of multidimensional integrals}, in Software for Numerical Mathematics, D.~J.
  Evans, ed., Academic Press, New York, 1974, pp.~159--172.

\bibitem{GG98}
{\sc T.~Gerstner and M.~Griebel}, {\em Numerical integration using sparse
  grids}, Numer. Algorithms,  (1998), pp.~209--232.

\bibitem{GG03}
\leavevmode\vrule height 2pt depth -1.6pt width 23pt, {\em Dimension-adaptive
  tensor-product quadrature}, Computing,  (2003), pp.~65--87.

\bibitem{GLSS07}
{\sc M.~Gnewuch, R.~Lindloh, R.~Schneider, and A.~Srivastav}, {\em Cubature
  formulas for function spaces with moderate smoothness}, J.~Complexity, 23
  (2007), pp.~828--850.

\bibitem{GW19}
{\sc M.~Gnewuch and M.Wnuk}, {\em Explicit error bounds for randomized
  {S}molyak algorithms and application to infinite-dimensional integration}.
\newblock Preprint, arXiv 1903.02276, 2019.
\newblock To appear in Journal of Approximation Theory, Volume 251, March 2020.

\bibitem{Gri06}
{\sc M.~Griebel}, {\em Sparse grids and related approximation schemes for
  higher order problems}, in Foundations of Computational Mathematics,
  Santander 2005, L.~M. Pardo, A.~Pinkus, E.~S{\"u}li, and M.~J. Todd, eds.,
  Cambridge, 2006, Cambridge University Press, pp.~106--161.

\bibitem{HHPS18}
{\sc A.-L. Haji-Ali, H.~Harbrecht, M.~D. Peters, and M.~Siebenmorgen}, {\em
  Novel results for the anisotropic sparse grid quadrature}, J. Complexity, 47
  (2018), pp.~62--85.

\bibitem{HHY03}
{\sc S.~Heinrich, F.~Hickernell, and R.~Yue}, {\em Optimal quadrature for
  {H}aar wavelet spaces}, Mathematics of Computation, 73 (2003), pp.~259--277.

\bibitem{HM11}
{\sc S.~Heinrich and B.~Milla}, {\em The randomized complexity of indefinite
  integration}, J. Complexity, 27 (2011), pp.~352--382.

\bibitem{Mat10}
{\sc J.~Matou\v{s}ek}, {\em Geometric Discrepancy}, Springer-Verlag Berlin
  Heidelberg, 2010.

\bibitem{Nie87}
{\sc H.~Niederreiter}, {\em Point sets and sequences wit small discrepancy},
  Monatsh. Math., 104 (1987), pp.~273--337.

\bibitem{NR96}
{\sc E.~Novak and K.~Ritter}, {\em High dimensional integration of smooth
  functions over cubes}, Numer. Math., 75 (1996), pp.~79--97.

\bibitem{NW10}
{\sc E.~Novak and H.~Wo\'zniakowski}, {\em Tractability of {M}ultivariate
  {P}roblems. Vol. 2: {S}tandard {I}nformation for {F}unctionals}, EMS Tracts
  in Mathematics, European Mathematical Society (EMS), Z\"urich, 2010.

\bibitem{Owe95}
{\sc A.~B. Owen}, {\em Randomly permuted $(t,m,s)$-nets and $(t,s)$-sequences},
  in Monte Carlo and Quasi-Monte Carlo Methods in Scientific Computing,
  H.~Niederreiter and P.~J.-S. Shiue, eds., New York, 1995, Springer,
  pp.~299--317.

\bibitem{Owe97}
\leavevmode\vrule height 2pt depth -1.6pt width 23pt, {\em {M}onte {C}arlo
  variance of scrambled equidistribution quadrature}, SIAM J. Numer. Anal., 34
  (1997), pp.~1884--1910.

\bibitem{Pet03}
{\sc K.~Petras}, {\em Smolyak cubature of given polynomial degree with few
  nodes for increasing dimension}, Numer. Math., 93 (2003), pp.~729--753.

\bibitem{PW04}
{\sc L.~Plaskota and G.~W. Wasilkowski}, {\em Smolyak's algorithm for
  integration and l1-approximation of multivariate functions with bounded mixed
  derivatives of second order}, Numerical Algorithms, 36 (2004).

\bibitem{Smo63}
{\sc S.~A. Smolyak}, {\em Quadrature and interpolation formulas for tensor
  products of certain classes of functions}, Dokl. Akad. Nauk. SSSR 4, 4
  (1963), pp.~240--243.

\bibitem{Tem92}
{\sc V.~N. Temlyakov}, {\em On a way of obtaining lower estimates for the
  errors of quadrature formulas}, Math. USSR Sbornik, 71 (1992), pp.~247--257.

\bibitem{WW95}
{\sc G.~W. Wasilkowski and H.~Wo\'zniakowski}, {\em Explicit cost bounds for
  algorithms for multivariate tensor product problems}, J. Complexity, 11
  (1995), pp.~1--56.

\end{thebibliography}

\end{document}